\pdfoutput=1
\documentclass[11pt]{article}
\usepackage[utf8]{inputenc}
\usepackage[T1]{fontenc}
\usepackage{lmodern}
\usepackage{geometry}
\author{David Forsman\\
Université catholique de Louvain\\
\texttt{david.forsman@uclouvain.be}~}
\title{Semi-abelian by Design: Johnstone Algebras Unifying Implication and Division}
\date{April 28, 2025}

\usepackage{graphicx}
\usepackage{amsthm, amssymb, mathtools}
\usepackage{enumitem}
\usepackage{multicol}
\usepackage{textcomp}
\usepackage{array}

\usepackage{tikz-cd}
\usepackage{caption, subcaption}
\usepackage{microtype}
\usepackage[style=numeric, backend=biber, giveninits=true]{biblatex}
\usepackage{xurl}
\usepackage{hyperref}
\hypersetup{
    breaklinks=true,   
    colorlinks=true,   
    pdfusetitle=true,  
 }
 \hypersetup{
    colorlinks=true,   
    linkcolor=blue,    
    citecolor=blue,    
}
 \usepackage{cleveref}

\theoremstyle{plain}
\newtheorem{theorem}{Theorem}[section]

\newtheorem{corollary}[theorem]{Corollary}

\theoremstyle{definition}
\newtheorem{definition}[theorem]{Definition}
\newtheorem{example}[theorem]{Example}
\newtheorem{remark}[theorem]{Remark}

\numberwithin{equation}{section}

\newcommand{\N}{\mathbb{N}}

\begin{document}
\maketitle

\begin{abstract}
Johnstone demonstrated that Heyting semilattices form a semi-abelian category via a specific triple of terms. Inspired by this work, we introduce \emph{Johnstone algebras} or J-algebras. The algebraic $(*,\to,e)$-theory $J$ of arities $(2,2,0)$ consists of three axioms carefully chosen to ensure protomodularity in alignment with Johnstone's terms. Johnstone algebras generalize well-known structures such as groups (division) and Heyting semilattices (implication) providing a unified framework within the well-behaved setting of semi-abelian categories.

We present two primary contributions. First, we identify the M-axiom,
\[
(t(x,y)\to x)\to (t(x,y)\to z) \approx x\to z, \text{ where }t(x,y) = (x\to y)\to y.
\]
The M-axiom is satisfied by residuated Johnstone algebras, and it can be considered a weakening of the H-axiom to comparable elements. We show that $t(x,y)$ defines a \emph{relative closure term} in MBC-algebras, and it implies that MBC-algebras form a variety of algebras, thereby generalizing the corresponding theorem related to HBCK-algebras. Second, we prove several no-go results, demonstrating that balanced theories or theories admitting non-discrete monotone or inflationary algebras cannot possess Malcev terms.

Together, these results establish Johnstone algebras as significant structures that achieve desirable categorical properties by carefully integrating both logical and symmetric features, while closely avoiding the constraints imposed by our no-go results.
\end{abstract}

\section*{Introduction}
Johnstone showed in \cite{Johnstone2004} that Heyting semilattices form a semi-abelian category \cite{Janelidze2001}, using a specific triple of $(\to,\land,\top)$-terms. These terms also demonstrate the semi-abelian nature of hoops \cite{Lapenta2024}, motivating a broader investigation into the core algebraic properties used. To escape the lattice-theoretic connotations, we consider the signature $(\to,*,e)$ of arities $(2,2,0)$.

The \emph{Johnstone protomodular terms} in this setting are given by:
\[
\theta(x,y,z) \coloneqq (y\to x) * z,\quad
\theta_1(x,y) \coloneqq x\to y,\quad
\theta_2(x,y) \coloneqq ((x\to y)\to y)\to x.
\]
We define the theory $J$ of \emph{Johnstone algebras} via axioms that ensure these terms satisfy the protomodularity equations. In the following, the equations on the left are for protomodularity and the equations on the right define the theory $J$ of Johnstone algebras:
\begin{multicols}{2}
\begin{enumerate}

    \item $\theta_1(x,x) \approx e$
    \item $\theta_2(x,x) \approx e$
    \item $\theta(y,\theta_1(x,y),\theta_2(x,y)) \approx x$
    \item[] $x\to x \approx e$\hfill (Reflexivity)
    \item[] $e\to x \approx x$\hfill (Unit-reduction)
    \item[] $\bar{x} *(\bar{x}\to x)\approx  x$\hfill (Monotone division)
\end{enumerate}
\end{multicols}
Here $\bar{x}$ is shorthand for the term $(x\to y)\to y$, which can be thought of as a closure of $x$ with respect to $y$, generalizing double negation. Notice that the two equations are the same on the first and the third row (Reflexivity, Monotone division). The two equations on the second row are different. This strengthening of $\theta_2(x,x) \approx e$ into $e\to x\approx x$ is only slight, since Unit-reduction in $J$ could be equivalently replaced by the two equations $\theta_2(x,x)\approx x$ and $x*e\approx x$. This equivalence can be seen by subtituting $x$ for $y$ in Monotone division.

A key feature of related logical algebras (such as BCK-algebras, and Heyting semilattices) is an order relation defined equationally via $x \leq y \coloneqq x \to y \approx e$. In $J$-algebras, this order is reflexive and anti-symmetric. The term $t(x,y) \coloneqq (x\to y)\to y$ plays an essential role in this framework. We investigate when an equation $x\leq y$ is \emph{strongly determined} by $t(x,y)$ in an algebraic theory $T$ in the following sense:
\begin{align*}
    T &\vdash x\leq t(x,y),\\
    T &\vdash x\leq y \Rightarrow t(x,y) \approx y.
\end{align*}
Strong determination allows the simplification of some implicative formulas to equations without changing the strength. For example, transitivity, $(x\leq y\land y\leq z)\Rightarrow x\leq z$, is equivalent to $x\leq t(t(x,y),z)$ in $T$, when the equation $x\leq y$ is strongly determined by the term $t(x,y)$. The quasi-equation of anti-symmetry requires more delicate handling. Specifically, we examine the equation $v(x,y) \approx t(v(x,y), y)$, where $v(x,y) \coloneqq t(t(x,y), x)$. This equation, which implies anti-symmetry, appears naturally in our derivations and is closely related to the Cornish condition J, $v(x,y)\approx v(y,x)$, known to characterize anti-symmetry equationally for HBCK-algebras~\cite{Bloki1993}.

This paper is structured as follows. \textbf{Section 1} reviews universal algebra preliminaries and formalizes the equational characterization of order structures. Theorem \ref{thm: Anti-symmetry equationalized} establishes that a theory admitting a \emph{weak relative closure term} $t(x,y)$, determining the equation $x\leq y$, allows one to equivalently replace the anti-symmetry $(x\leq y\leq x)\Rightarrow x\approx y$ with the equation $v(x,y)\approx t(v(x,y),y)$. This shows, with Theorem \ref{thm: MC-algebras have relative closure term} that the quasi-variety of MBC-algebras forms a variety, thereby generalizing Corollary 4.8 in~\cite{Bloki1993} and the main result in~\cite{Kowalski1994}. 

\textbf{Section 2} situates Johnstone algebras within categorical algebra, defining Malcev, protomodular, and semi-abelian theories. Here, we present several \emph{no-go theorems} that identify conditions under which semi-abelian properties are impossible without triviality: If a theory $T$ is balanced or admits a partially ordered, non-discrete algebra in which each structure map is monotonic (monotone algebra) or each structure map is inflationary in some component (inflationary algebra), then $T$ cannot have a Malcev term. We conclude by exploring compositional (B) and residual (R) Johnstone algebras $(M,*,\to,e)$, the axioms of which ensure that $\leq$ is a partial order in $M$, $(M,*,e)$ is a monoid and there is a Galois connection $x * (-) \dashv x \to (-)$ for each $x\in M$. The theory of RBCJ-algebras, or residual, compositional and commutative Johnstone algebras, generalizes the theories of commutative groups and hoops.

\section{Equational Order}
This section, after preliminaries of the standard universal algebra, focuses on equationally defined orders. We are particularly interested in when an equation $x \le y$ is \emph{strongly determined} by a term $t(x,y)$:
$$
T \vdash x \le y \implies t(x,y) \approx y\quad \text{ and }\quad T \vdash x \le t(x,y).
$$
Strong determination converts formulas $(x \le y \implies E(x,y))$ into less implicative form $E(x, t(x,y))$, simplifying properties like transitivity. While anti-symmetry often remains a quasi-equation, we prove it is equivalent to the equation $v(x,y) \approx t(v(x,y),y)$ (with $v(x,y)=t(t(x,y),x)$) if the term $t(x,y)$ is a \emph{weak relative closure term}, generalizing \cite{Kowalski1994}.
\begin{definition}[Signatures, Terms, and Equations]
Let $V$ be a fixed countably infinite set of variables $x_i,i\in \N$. A \emph{signature} $\sigma$ is a pair $(F, \mathrm{arity}\colon F\to\N)$, where $F$ is a set of function symbols and $\mathrm{arity}$ assigns to each symbol its arity. Elements $f\in F$ with $\mathrm{arity}(f)=0$ are called \emph{constant symbols}. 

The set $\mathrm{Term}$ of \emph{$\sigma$-terms} is defined as the smallest set satisfying:
\begin{enumerate}
    \item Every variable in $V$ is a term.
    \item Every constant symbol is a term.
    \item If $f\in F$ has non-zero arity $n$ and $t_1,\ldots,t_n\in\mathrm{Term}$, then $f(t_1,\ldots,t_n)$ is a term. If $n = 2$, we often denote $t_1ft_2$ for $f(t_1,t_2)$.
\end{enumerate}
A term $t$ may be denoted by $t(x_1,\ldots,x_n)$ to indicate that the set of variables occurring in $t$ is a subset of $\{x_1,\ldots,x_n\}\subset V$. Given a term $t(x_1,\ldots,x_n)$ and a function $s\colon V\to\mathrm{Term}$, we define the substitution
\[
t(s(x_1),\ldots,s(x_n))\coloneqq \bar{s}(t),
\]
where $\bar{s}\colon \mathrm{Term}\to\mathrm{Term}$ is the unique extension of $s$ that commutes with the operations of $\sigma$.

A \emph{$\sigma$-equation} is an expression of the form $t_1\approx t_2$ with $t_1,t_2\in\mathrm{Term}$. An equation is called \emph{balanced} if both sides of the equations have the same variables expressed. If $T$ is a set of equations, we call $T$ an \emph{algebraic $\sigma$-theory} and \emph{balanced} in the case all the equations of $T$ are balanced. We say that $T$ \emph{deduces} or \emph{proves} an equation $\phi$, written $T\vdash\phi$, if $\phi$ belongs to the smallest set $D_T$ of equations satisfying:
\begin{enumerate}
    \item Every equation in $T$ is in $D_T$. \hfill (Extension)
    \item For any $\sigma$-term $t$, the identity equation $t\approx t$ is in $D_T$. \hfill (Reflexivity)
    \item If $t\approx s\in D_T$, then $s\approx t\in D_T$.\hfill (Symmetry)
    \item If $p\approx q\in D_T$ and $q\approx r\in D_T$, then $p\approx r\in D_T$. \hfill (Transitivity)
    \item If $t_1(x_1,\ldots,x_n)\approx t_2(x_1,\ldots,x_n),s_i\approx s_i'\in D_T$, for $i\leq n$, then
    \begin{align*}
        t_1(s_1,\ldots,s_n)\approx t_2(s'_1,\ldots,s'_n)\in D_T.\tag{Substitution}
    \end{align*}
\end{enumerate}
\end{definition}

\begin{definition}[Algebras and Homomorphisms]
Let $\sigma=(F, \mathrm{arity})$ be a signature and $V$ a set of variables.
A \emph{$\sigma$-model} $M$ consists of a carrier set, also denoted $M$, and, for each $n$-ary function symbol $f \in F$, an $n$-ary operation $M(f)\colon M^n \to M$ also called \emph{structure map} if $n>0$ and a \emph{structure constant} if $n = 0$.

An \emph{assignment} into $M$ is a function $s\colon V \to M$. Any assignment extends uniquely to an evaluation map $\overline{s}\colon \mathrm{Term} \to M$ such that $\overline{s}(x) = s(x)$ for $x \in V$, and
\[
\overline{s}(f(t_1,\ldots,t_n)) = M(f)(\overline{s}(t_1),\ldots,\overline{s}(t_n))
\]
for $f \in F$ of arity $n$ and terms $t_1, \dots, t_n$. We denote the induced function as $M(t)\colon M^V\to M$, $s\mapsto \overline{s}(t)$.

A $\sigma$-model $M$ \emph{satisfies} a $\sigma$-equation $t_1 \approx t_2$, denoted $M \vDash t_1 \approx t_2$, if $M(t_1) = M(t_2)$.
If $M$ satisfies all equations in a theory $T$, we call $M$ a \emph{$T$-algebra}. If all $T$-algebras satisfy an equation $\phi$, we write $T \vDash \phi$ and say $\phi$ is a \emph{semantic consequence} of $T$.

Given $T$-algebras $M$ and $N$, a function $g\colon M \to N$ is a \emph{$T$-algebra homomorphism} if for every $k$-ary $f \in F$ and all $a_1, \ldots, a_k \in M$,
\[
g(M(f)(a_1,\ldots,a_k)) = N(f)(g(a_1),\ldots,g(a_k)).
\]
The category of $T$-algebras and their homomorphisms is denoted by $T$-\textbf{Alg}. A category $C$ is called a \emph{variety} of algebras if it is equivalent to $T$-\textbf{Alg} for some algebraic theory $T$.
\end{definition}

\begin{example}[Examples of Algebraic Theories]\label{Examples of theories}
Let $\sigma = (*, \to, e)$ be a signature with arities $(2, 2, 0)$. We define the terms $t(x,y) \coloneqq (x\to y)\to y$ and $v(x,y) \coloneqq t(t(x,y),x)$. Consider the following equations, which serve as potential axioms for various algebraic theories:
\begin{multicols}{2}\footnotesize
\begin{enumerate}
    \item $(x * y) * z \approx x * (y * z)$ \hfill (Associativity)
    \item $x * e \approx x$ \hfill (Right identity)
    \item $e*x \approx x$ \hfill (Left identity)
    \item $x * y \approx y * x$ \hfill (Commutativity)

    \item $x\to x \approx e$ \hfill (Reflexivity)
    \item $e\to x \approx x$ \hfill (Unit-reduction)
    \item $x*(x\to y)\approx y*(y\to x)$ \hfill (Hoop-symmetry)
    \item $x\to (y\to z) \approx (y*x)\to z$ \hfill (Residuation)
    \item $x*x\approx x$\hfill (Idempotency)
    \item $x*y \approx y*(y\to x)$ \hfill (Fusion)
    \item $x\to e \approx e$\hfill (Unit-absorption / T-axiom\footnote{T-axiom (Top/Truth) is not a standard in name combinatory logic, it follows from Reflexivity and the K-axiom. The letter T references to top and truth and so it motivates the name.})

    \item $x*(x\to y) \approx y$ \hfill (Lower division)
    \item $x\to (x*y) \approx y$ \hfill (Upper division)

    \item $x\to (y\to x)\approx e$\hfill (Projection / K-axiom)
    \item $x\to t(x,y) \approx e$\hfill (Modus ponens)
    \item $x\to (y\to z)\approx y\to (x\to z)$ \hfill \\(Implicative commutativity / C-axiom)
    \item $(x\to y)\to ((z\to x)\to (z\to y)) \approx e$\hfill \\(Compositionality / B-axiom)
    \item $(z\to (x\to y))\to ((z\to x)\to (z\to y)) \approx e$ \hfill (Self-distributivity / S-axiom)
    \item $t(x,y)* (t(x,y)\to x) \approx x$ \hfill (Monotone division)
    \item $(x\to y)\to (x\to z ) \approx (y\to x)\to (y\to z)$\hfill (Conditional exchange / H-axiom)
    \item $(t(x,y)\to x)\to (t(x,y)\to z) \approx (x\to z)$ \hfill (Monotone exchange /M-axiom\footnote{M-axiom is a new notion and an implicative version of the Monotone division (19). In the presence of Unit-reduction and Modus ponens, the $M$-axiom is equivalent to the H-axiom holding for comparable elements $x,y$.})
    \item $v(x,y) \approx v(y,x)$ \hfill (Cornish condition J)
\end{enumerate}
\end{multicols}
When defining a theory using a subset of the equations above, we implicitly assume the signature contains only the necessary function symbols. Below are examples of theories defined using these axioms.
    \begin{enumerate}[label=(\alph*)]
    \item \textbf{Pointed Semilattices:} The theory defined by equations (1)--(4) and (9) (associativity, identities, commutativity, idempotency). These algebras are equivalent to idempotent commutative monoids with identity $e$.

    \item \textbf{Left-Loops and Groups:} The theory of \emph{left-quasigroups} is defined by the division axioms (12) and (13). Adding the identity laws (2) and (3) defines the theory of \emph{left loops}. If associativity (1) is also added, we obtain the theory of \emph{groups}.

    \item \textbf{Hoops and Heyting Semilattices:} The theory of \emph{hoops} comprises of equations (1)--(8) (axioms for a commutative monoid structure with Reflexivity, Unit-reduction, Hoop-symmetry, and Residuation). Adding Fusion (10) yields the theory of \emph{Heyting semilattices}. In any hoop $H$, the following conditions are equivalent for an element $a \in H$:
    \begin{itemize}
        \item $a*a = a$ \hfill (Idempotency)
        \item $a*b= b*(b\to a)$ for all $b\in H$ \hfill (Fusion)
        \item $(a\to (b\to c))\to ((a\to b)\to (a\to c)) = e$ for all $b,c\in H$ \hfill (Self-distributivity)
    \end{itemize}
    The relation $a\leq b \iff a\to b = e,$ for $a,b\in H$, defines a partial order on $H$ and the meet $a\land b$ is equal to $a*(a\to b)$ for $a,b\in H$.

    \item \textbf{Implicative algebras:} These are algebras over the signature $(\to,e)$ satisfying:
    \begin{itemize}
        \item Reflexivity (5) and Unit-reduction (6).
        \item Implicative anti-symmetry: $(x\to y \approx e \land y\to x \approx e) \implies x \approx y$.
        \item Implicative monotonicity: $y\to z \approx e \implies (x\to y)\to (x\to z) \approx e$.
    \end{itemize}
    Different classes of implicative algebras are often denoted by concatenating letters from $\{H, M,  S, B, C, K, T\}$ corresponding to axioms (20), (21), (18), (17), (16), (14), (11), respectively. For instance, HBCK-algebras satisfy (20), (17), (16), and (14). When Modus ponens (15) holds (i.e., $x \le t(x,y)$), Implicative monotonicity is equivalent to the equation $(x\to y)\to (x\to t(y,z)) \approx e$, which becomes $(y\to z)\to ((x\to y)\to (x\to z)) \approx e$, the B-axiom, after a double application of Implicative commutativity. This shows that the $B$-axiom is equivalent to the Implicative monotonicity, given Implicative commutativity and hence C-algebras and BC-algebras coincide. Notice how the relation $a\leq b\iff a\to b = e; a,b\in M$, defines a partial order in each implicative algebra $M$.
    
    It is known that HBCK-algebras are precisely the $(\to,e)$-subreducts of hoops \cite{Bloki1993}. This implies they satisfy the Cornish condition J (22), which in turn allows the Implicative anti-symmetry axiom to be replaced by the equation $v(x,y) \approx v(y,x)$, making the class of HBCK-algebras a variety. Kowalski provided a direct syntactic derivation of the Cornish condition J from the quasi-equational HBCK-axioms \cite{Kowalski1994}. Motivated by Kowalski's work, we demonstrate in Corollary \ref{Algebras forming varieties} that MC-algebras (equivalently MBC-algebras) form a variety of algebras. It has been shown that $L$-algebras, defined as HT-algebras, do not form a variety \cite{facchini2023idealscongruenceslalgebrasprelalgebras}. Neither do the BCK-algebras \cite{Wronski1983}. Corollary \ref{Algebras forming varieties} shows that $L$-algebras satisfying the $C$-axiom, or HCT-algebras, do form a variety of algebras, but these are exactly the same as HBCK-algebras.

    \item \textbf{Johnstone Algebras:} The $(*,\to, e)$-theory $J$ of Johnstone algebras consists of axioms Reflexivity, Unit-reduction and Monotone division (5,6,19). The theory $J$ is constructed to ensure that the Johnstone protomodular terms (defined in Section 2) satisfy the protomodularity equations, thereby guaranteeing that the category $J$-\textbf{Alg} is semi-abelian (as it is pointed). The axioms imply Right identity (2) and that the relation $x\leq y \iff x\to y \approx e$ is reflexive and anti-symmetric.
    Prominent examples like the theories of left loops, groups, hoops, and Heyting semilattices are all extensions of $J$. For instance, hoops satisfy (5) and (6) by definition. They also satisfy Monotone division (19), as demonstrated by the following derivation (using properties derivable in hoops):
    \begin{align*}
        t(x,y)* (t(x,y)\to x)
        &\approx x *(x\to t(x,y)) \tag{Hoop symmetry}\\
        &\approx x*e  \tag{Modus ponens}\\
        &\approx x  \tag{Right identity}
    \end{align*}
    Notice that one can deduce from Residuation and Monotone division the M-axiom:
    $$
    (\bar{x}\to x)\to (\bar{x}\to z)\approx (\bar{x}*(\bar{x}\to x))\to z\approx x\to z
    $$
    The M-axiom can be seen as a weakening of Conditional Exchange, H-axiom (20), and acts as a substitution rule for equivalent terms, in implicative algebras that satisfy Modus ponens, $x\leq t(x,y)$.
    \end{enumerate}
\end{example}\begin{definition}[Term Algebra]
Let $T$ be a $\sigma$-theory. The set $\mathrm{Term}$ of $\sigma$-terms naturally becomes a $\sigma$-model by defining
\[
\mathrm{Term}(f)\colon \mathrm{Term}^n\to\mathrm{Term},\quad (t_1,\ldots,t_n)\mapsto f(t_1,\ldots,t_n),
\]
for each $f\in F$ of arity $n$. Define the equivalence relation $\sim$ on $\mathrm{Term}$ by 
\[
t_1\sim t_2\quad\Longleftrightarrow\quad T\vdash t_1\approx t_2.
\]
The quotient $\mathrm{Term}/_\sim$, equipped with the induced operations, is called the \emph{term algebra} of $T$. The elements of $\mathrm{Term}/_\sim$ are denoted by $[t]$ for $t\in \mathrm{Term}$. 
\end{definition}

\begin{theorem}[Soundness and Completeness]\label{Soundness and completeness}
    Let $T$ be an algebraic $\sigma$-theory. Then $T\vdash \phi$ if and only if $T\vDash \phi$ for any $\sigma$-equation $\phi$. Especially, the term algebra of $T$ satisfies exactly those equations provable from $T$.
\end{theorem}
\begin{proof}
    The forward direction, soundness, is proven by induction. Let $M$ be a $T$-algebra. Consider the set $I$ of the $\sigma$-equations $M$ satisfies. We show $D_T\subset I$. Clearly, $I$ is reflexive, anti-symmetric and transitive. For the substitution condition, assume that $t_1(x_1,\ldots, x_N)\approx t_2(x_1,\ldots, x_n),s_i\approx s_i'\in I$ for $i\leq n$. Now for $S\colon V\to M$ we have 
    \begin{align*}
    M(t_1(s_1,\ldots, s_n)) 
    &= t_1(M(s_1),\ldots, M(s_n)) \tag{Structural induction on terms}\\
    &= t_2(M(s_1),\ldots, M(s_n)) \tag{$t_1\approx t_2\in I$}\\
    &= t_2(M(s_1'),\ldots, M(s_n'))\tag{$s_i\approx s_i'\in I,i\leq n$}\\
    &= M(t_2(s_1',\ldots, s_n')),
    \end{align*}
    where we suppressed the notation $M(t)(S)$ to $M(t)$, and so $t_1(s_1,\ldots, s_n)\approx t_2(s_1',\ldots, s_n')\in I$.

    We show the converse, completeness. Consider the associated term algebra $M = \mathrm{Term}/_{\sim}$, where $\sim$ is the provability relation. Notice that by the definition of deduction the function $\mathrm{Term}(f)\colon \mathrm{Term}^n\to \mathrm{Term}$ induces the function $M(f)\colon M^n\to M$ for function symbols $f$ in $\sigma$.. Thus $M$ is a $\sigma$-model. We need to see that $M$ satisfies the equations in $T$.
    
    Assume that $t_1(x_1,\ldots, x_n) \approx t_2(x_1,\ldots, x_n)\in T$. Let $s\colon V\to M$ and denote $s(x_i) = [s_i]$, where $s_i\in\mathrm{Term}$ for $i\in\N$. We need to show that $t_1([s_1],\ldots,[s_n]) = t_2([s_1],\ldots,[s_n])$, which by induction on the structure of terms is equivalent to $[t_1(s_1,\ldots, s_n)]=[t_2(s_1,\ldots, s_n)]$. Since $t_1(x_1,\ldots, x_n)\approx t_2(x_1,\ldots,x_n)\in T$, then by the substitution rule $T\vdash t_1(s_1,\ldots, s_n)\approx t_2(s_1,\ldots, s_n)$. Hence $[t_1(s_1,\ldots, s_n)] = [t_2(s_1,\ldots,s_n)]$ and so $M\vDash t_1\approx t_2$. Thus $M$ is a $T$-algebra. For completeness, assume that $M\vDash t_1(x_1,\ldots, x_n)\approx t_2(x_1\ldots, x_n)$. Thus $t_1([x_1],\ldots,[x_n]) = t_2([x_1],\ldots,[x_n])$ and therefore $[t_1(x_1,\ldots, x_n)] = [t_2(x_1,\ldots, x_n)]$ and hence $T\vdash t_1\approx t_2$. 
\end{proof}
\begin{definition}[Deducing Quasi-equations]\label{def:DeducingQuasi}
    Let $\sigma$ be a signature. We call a symbol $c$ independent of $\sigma$, if $c$ is neither a variable nor a function symbol of $\sigma$. Let $T,T'$ be $\sigma$–theories and let $P(x_1,\ldots, x_n)$ be a $\sigma$-equation. We denote $T\vdash \bigwedge T'\Rightarrow P(x_1,\ldots, x_n)$ and say that $T$ proves that $P$ is entailed by $T'$ if $T\cup\{Q(a_1,\ldots, a_m)\mid Q(x_1,\ldots, x_m)\in T'\}\vdash P(a_1,\ldots, a_n)$ where the symbols $a_i$ for $i\in\N$ are pairwise different symbols independent from $\sigma$. The theories $T$ and $T'$ are considered $\sigma'$ theories, where $\sigma'$ extends $\sigma$ by the addition of the constant symbols $a_i,i\in\N$. If $T = \{Q_1,\ldots, Q_n\}$, we will often denote $\bigwedge T$ as $(Q_1\wedge\cdots \wedge Q_n)$.
\end{definition}

\begin{theorem}
    Let $T$ and $T'$ be $\sigma$-theories. Let $P(x_1,\ldots, x_n)$ be a $\sigma$-equation. Then $T\vdash \bigwedge T'\Rightarrow P$ if and only if for all $T$–algebras $M$ the equation $P(r_1,\ldots, r_n)$ holds for those $r_i\in M,i\in\N$, where $Q(r_1,\ldots, r_m)$ holds for each $Q(x_1,\ldots, x_m)\in T'$.
\end{theorem}
\begin{proof}
    The proof is a simple application of the Soundness and Completeness Theorem \ref{Soundness and completeness}.
\end{proof}
\begin{definition}
    Let $T$ be an algebraic $\sigma$-theory. Let $P(x,y)$ be a $\sigma$-equation. We say that the equation $P(x,y)$ defines an \emph{equational $T$-order}, if 
    \begin{align*}
        T\vdash& P(x,x) \tag{$P$-Reflexivity}\\
        T\vdash& (P(x,y)\wedge P(y,x))\Rightarrow x \approx  y\tag{$P$-Anti-symmetry}\\
        T\vdash& (P(x,y)\wedge P(y,z))\Rightarrow P(x,z)\tag{$P$-Transitivity}
    \end{align*}
    If $P(x,y)$ defines an equational $T$-order, we often denote the equation $P(x,y)$ by $x\leq y$. 
    We say that the equation $P$ is semi-determined by a $\sigma$-term $t$, if
    $$
    T\vdash P(x,y)\Rightarrow t(x,y) \approx  y.
    $$
    If in addition
    $$
    T\vdash t(x,y)\approx y\Rightarrow P(x,y),
    $$
    then we say that the equation $P$ is \emph{determined} by the term $t$. Lastly, we call $P$ \emph{strongly determined} by $t$, if $P$ is semi-determined by $t$ and $T\vdash P(x, t(x,y))$. 
\end{definition}

\begin{theorem}\label{thm: characterizing reflectivity and transitivity}
    Let $T$ be an algebraic $\sigma$-theory. Consider a $\sigma$-equation $x\leq y$ semi-determined by a $\sigma$-term $t(x,y)$. Denote $v(x,y) = t(t(x,y),x)$. Then the following holds:
    \begin{enumerate}
        \item Let $E(x,y)$ be a $\sigma$-equation. Then 
        $$
        T\vdash E(x,t(x,y)) \text{ implies } T\vdash x\leq y\Rightarrow E(x,y) 
        $$
        
        Moreover, the converse holds, if the equation $x\leq y$ is strongly determined by the term $t(x,y)$. 

        \item If $x\leq y$ is determined by $t(x,y)$, then $\leq$ is reflexive if and only if $T\vdash t(x,x) \approx x$.
        \item If $T$ proves either $v(x,y) \approx v(y,x)$ or $v(x,y) \approx t(v(x,y),y)$, then $\leq$ is anti-symmetric.
        \item Transitivity of $\leq$ is implied by $x\leq t(t(x,y),z)$ and the converse holds, if $\leq$ is strongly determined by $t$. 
    \end{enumerate}
\end{theorem}
\begin{proof}
    \hfill
    \begin{enumerate}
        \item The forward direction is clear, since $x\leq y$ is semi-determined by $t(x,y)$.
            For the partial converse, assume that $T\vdash x\leq y\Rightarrow E(x,y)$. Let $M$ be a $T$-algebra. Let $m,n\in M$. By completeness, it suffices to show that $E(m,t(m,n))$. Since $\leq$ is strongly determined by $t$, $m\leq t(m,n)$ holds and thus by soundness $E(m,t(m,n))$. 
        \item Clear.
        \item Let $a$ and $b$ be symbols independent from $\sigma$. The theory $T\cup \{a\leq b, b\leq c\}$ proves that 
        $$
        v(a,b) = t(t(a,b), a) \approx t(b,a)\approx a
        $$
        and thus if $T\vdash v(x,y)\approx v(y,x)$ it directly follows that $T\vdash (x\leq y\land y\leq x)\Rightarrow x\approx y$. Similarly with the equation $v(x,y) \approx t(v(x,y), y)$.

        \item Assume that $T\vdash x\leq t(t(x,y),z)$. Now the theory $T\cup\{a\leq b, b\leq c\}$ proves that $t(a,b) \approx b$ and $t(b,c) \approx c$ by the fact that $\leq$ is semi-determined by $t$. Thus we attain that $T\cup\{a\leq b, b\leq a\}$ proves the the equations $a\leq t(t(a,b),c), a\leq t(b,c)$ and hence $a\leq c$.
        
        For the partial converse, assume that $\leq$ is strongly determined by $t$. Consider the following sequence of statement. 
        \begin{align*}
            T&\vdash (x\leq y\land y\leq z)\Rightarrow x\leq z\\
            T&\vdash (x\leq y)\Rightarrow  x\leq t(y,z)\\
            T&\vdash x\leq t(t(x,y),z)
        \end{align*}
        The argument used in part $1$ applies here and shows that these steps are equivalent, since $\leq$ is strongly determined by $t$. This shows the partial converse.\qedhere
        \end{enumerate}
\end{proof}

While Theorem \ref{thm: characterizing reflectivity and transitivity}(4) allows replacing the transitivity quasi-equation with an equation when the order $\leq$ is strongly determined by a term $t$, the anti-symmetry quasi-equation $(x\leq y\land y\leq x \implies x \approx y)$ generally resists such equational reduction. BCK-algebras illustrate this: although $x \leq y = (x \to y \approx e)$ is strongly determined by $t(x,y)=(x\to y)\to y$, anti-symmetry remains a quasi-equation, preventing BCK-algebras from forming a variety \cite{Wronski1983}.

However, Theorem~\ref{thm: characterizing reflectivity and transitivity} (3) identifies an equation sufficient for anti-symmetry: $v(x,y) \approx t(v(x,y),y)$, where $v(x,y)\coloneqq t(t(x,y),x)$. A heuristic motivation for this equation arises from applying the substitution principle (Theorem \ref{thm: characterizing reflectivity and transitivity}(1)) to the anti-symmetry quasi-equation, assuming strong determination:
\begin{align*}
    (x\leq y \land y\leq x) &\implies (x \approx y) \\
    \iff (t(x,y) \leq x&\implies x \approx t(x,y)) \tag{$x\leq t(x,y)$} \\
    \leadsto v(x,y) &\approx t(v(x,y),y) \tag{heuristic substitution; $t(t(x,y),x)$ substituted on $x$}
\end{align*}
While the first equivalence is justified under strong determination, the step yielding $v(x,y) \approx t(v(x,y),y)$ is heuristic, not a rigorous deduction via the theorem alone. Nevertheless, it suggests this equation as a natural candidate for capturing anti-symmetry equationally.

\begin{corollary}\label{Order corollary}
    Let $\sigma$ be a signature and let $t(x,y)$ be a $\sigma$-term. Denote the equation $t(x,y) \approx y$ as $x\leq y$ and $v(x,y) = t(t(x,y),x)$ and consider the $t$-theory $T$.
    \begin{align*}
            t(x,x) &\approx x\tag{$t$-reflexivity}\\
            v(x,y)&\approx t(v(x,y),y)\tag{$t$-anti-symmetry}\\
            t(x,t(t(x,y),z)) &\approx t(t(x,y),z)\tag{$t$-transitivity}
    \end{align*}
    Then $\leq$ is a $T$-order.
\end{corollary}
\begin{proof}
    Follows from Theorem \ref{thm: characterizing reflectivity and transitivity}(2,3,4).
\end{proof}

\begin{definition}\label{def: closure term}
    Let $T$ be an algebraic $\sigma$-theory with a term $t(x,y)$ also denoted $xy$. We fix the following notation $x\leq y\coloneqq xy \approx y$, $\bar{x} \coloneqq xy, \bar{\bar{x}} = (xy)z$ and $X = (xy)x$. Consider the equations
    \begin{multicols}{2}\small
    \begin{enumerate}
        \item $x \approx xx$ \hfill ($t$-reflexivity, Idempotency)
        \item $x\bar{\bar{x}} \approx \bar{\bar{x}}$ \hfill ($t$-transitivity)
        \item $\bar{x}x \approx  (\bar{x}x)y$ \hfill ($t$-anti-symmetry)
        \item $x\bar{x} \approx \bar{x}$ \hfill (Left-absorption)
        \item $\bar{x}y \leq \bar{x}$\hfill (Right-absorption)
        \item $xz \leq \bar{x}z$\hfill (Left-monotonicity)
        \item $\bar{x}(xz) \approx \bar{x}z$ \hfill (Flattening)
        \item $(\bar{\bar{x}}\bar{x})x \leq \bar{\bar{x}}x$\hfill (Closure stability)
        \item $(X\bar{x})x \leq Xx$\hfill (Weak closure stability)
    \end{enumerate}
    \end{multicols}
    We call $t$ a \emph{pre-order term}, if $T$ proves $t$-reflexivity and $t$-transitivity. If in addition $t$ satisfies $t$-anti-symmetry, it is called an \emph{order term}. We call $t$ a \emph{relative closure} if $T$ proves the equations (4)--(8). If $T$ proves $(5),(7)$ and $(9)$, then we say that $t$ is a weak relative closure.
\end{definition}
\begin{remark}\label{Rem: Antisymmetry upgrade}
    If the term $t(x,y)$ contains both variables $x$ and $y$, then the equations (1)--(9) in Definition \ref{def: closure term} are balanced.
\end{remark}
Both the meet $x\land y$ and the join $x\lor y$ of a bounded lattice satisfies all the nine equations. The equations $(1)-(6)$ determine that the term $t(x,y)$ behaves like closure operator (defining a monad) for a fixed $y$ in the induced posetal structure. In this context, Left-absorption relates to Modus ponens and shows that $t$ strongly determines $\leq$. Perhaps surprisingly, the term $(x\to y)\to y$ of MC-algebras satisfies all the equations as per the following theorems:

\begin{theorem}\label{thm: MC-algebras have relative closure term}
    Let $M$ be $(\to,e)$-algebra satisfying Reflexivity, Unit-reduction and Monotone exchange (M-axiom). We define the relation $\leq$ by setting $a\leq b\iff a\to b = e$ for $a,b\in M$. Consider the term $t(x,y) = (x\to y)\to y$. 
    \begin{enumerate}
        \item The equation $x\leq y$ is semi-determined by $t(x,y)$ and $\leq$ is a pre-order on $M$. 
        \item Assume that $(-)\to (-)$ is increasing in the second component and $M$ satisfies Implicative commutativity. Then $\to$ is decreasing in the first component. In addition, the term $t$ is both a pre-order term and a relative closure term in $M$. 
    \end{enumerate}
\end{theorem}

\begin{proof}
    We will use a left associative notation that omits the implication arrow and where a space refers to a use of an implication arrow. For instance, $xyyxx$, $xy\ yx$ and $xyy$, we mean the terms $(((x\to y)\to y)\to x)\to x$, $(x\to y)\to (y\to x)$ and $(x\to y)\to y$, respectively.

    Unit-reduction implies that $x\le y$ is semi-determined by the term $t(x,y) = (x\to y)\to y$. By assumption, $\leq$ is reflexive. For transitivity consider the following:  Let $a\leq b\leq c$ in $M$. Now
        \begin{align*}
            ac
            &= (ba)(bc)\tag{Monotone exchange; $a\leq b$}\\
            &= (ba)e \tag{$b\leq c$}\\
            &= (ba)(bb)\tag{Reflexivity}\\
            &= ab\tag{Monotone exchange; $a\leq b$}\\
            &= e \tag{$a\leq b$}
        \end{align*}

    For the second part, assume that $M$ satisfies Implicative commutativity and $\to$ is increasing in the second component. Note that Implicative commutativity (C) together with Reflexivity directly imply Modus ponens (15). Thus $t$ strongly determines $\leq$ and so $t$ is a pre-order term by Theorem \ref{thm: characterizing reflectivity and transitivity}(4) and $t$ satisfies Left-absorption. We show that $\to$ is decreasing in the first component: Note that $M$ satisfies the B-axiom, $yz\ (xy\ xz) \approx e$, (Ex: \ref{Examples of theories}(d)) and by Implicative commutativity $M$ satisfies $xy\ (yz\ xz)\approx e$. This directly implies that $\to$ is decreasing in the first component. Hence $t(x,y)$ is increasing in $x$.

    Notice that Implicative commutativity yields that $(-)\to y\colon M^{op}\to M$ defines a mutual right self-adjoint for each $y\in M$. Thus $\bar{x}y\leq xy\leq\bar{x}y$ and especially, the Right-absorption $t(t(x,y),y)\leq  t(x,y)$ holds in $M$.
    
    Let $x,y, z\in M$ and denote $\bar{x} = xyy$ and $\bar{\bar{x}} = xyyzz$. For Flattening, we need to show that $(\bar{x}\ xzz)\ xzz = \bar{x} zz$. Now
    \begin{align*}
        (\bar{x}\ xzz)\ xzz 
        &= (xz\ \bar{x}z)\ xzz\tag{Implicative commutativity}\\
        &= \bar{x}zz \tag{Monotone exchange; $\bar{x}z\leq xz$}
    \end{align*}
    Lastly, we show that $M$ satisfies Closure stability. We may assume that $x\leq y\leq z$ and it suffices to show that $zx\leq \bar{z}x$, where $\bar{z} = zyy$, since this implies that $zyyxx\leq zxx$. Now
    \begin{align*}
        zx\ \bar{z}x
        &= (zy\ zx)(zy\ \bar{z}x)\tag{Monotone exchange; $zx\leq zy$}\\
        &\ge (zy\ zx)(\bar{z}y\ \bar{z}x)\tag{Self-adjointness of $(-)\to y$: $zy\leq \bar{z}y$}\\
        &= yx\ yx\tag{2xMonotone exchange; $y\leq z, y\leq \bar{z}$}\\
        &= e\tag{Reflexivity}
    \end{align*}
    Unit-reduction then implies that $zx\ \bar{z}x = e$ and thus $zx\leq \bar{z}x$ and hence $(\bar{\bar{x}}\bar{x})x \leq  \bar{\bar{x}}x$.\qedhere
\end{proof}

\begin{theorem}\label{thm: Anti-symmetry equationalized}
    Let $T$ be an algebraic $\sigma$-theory with a weak relative closure term $t(x,y)$, which we denote by $xy$. We set $x\leq y \coloneqq xy\approx y, \bar{x} \coloneqq xy, X \coloneqq (xy)x,\bar{y}\coloneqq yx$ and $Y\coloneqq (yx)y$. Let $M$ be a $T$-algebra. Then 
    $$
    M\vDash(x\leq y \land y\leq x)\Rightarrow x \approx y\text{ if and only if }M\vDash X \approx Xy
    $$
    Furthermore, $M\vDash X\approx Xy, y\leq xy$ yields the Cornish condition $J$, $(xy)x \approx (yx)y$, for $M$.
\end{theorem}
\begin{proof}
    The converse is clear. Assume $M\vDash(x\leq y \land y\leq x)\Rightarrow x \approx y$. Let $x,y,z\in M$. We upgrade the Right-absorption, $\bar{x}y\leq \bar{x}$, and the Closure stability, $(X\bar{x})x\leq Xx$, into equations. First, we have $(\bar{x}\bar{x})\bar{x} = (\bar{x}y)\bar{x} = \bar{x}$ and $\bar{x}(\bar{x}\bar{x}) = ((\bar{x}\bar{x})\bar{x})(\bar{x}\bar{x}) = \bar{x}\bar{x}$ by $1\times$Flattening and $2\times $Right-absorption. By anti-symmetry, $\bar{x} = \bar{x}\bar{x} = \bar{x}y$ holds. Second, if $(xy)z\leq x = xz$, then $x = (xy)z$: By Flattening and Right-absorption we have $x=((xy)z)x = ((xy)z)(xz) = ((xy)z)z = xyz$. This shows that the Closure stability, $(X\bar{x})x\leq Xx = X$ (by Right-absorption), strengthens to $(X\bar{x})x = X$. To proceed, note that $\bar{x}\leq \bar{x}z$, since 
    \begin{align*}
        \bar{x}(\bar{x}z)
        &= (\bar{x}y)(\bar{x}z)\tag{Right-absorption: $\bar{x} = \bar{x}y$}\\
        &= (\bar{x}y)z \tag{Flattening}\\
        &=  \bar{x}z \tag{Right-absorption: $\bar{x} = \bar{x}y$}
    \end{align*} and especially $X\leq Xy$. By Flattening and Right-absorption, the equation $X\bar{x} = (\bar{x}x)(\bar{x}y) = (\bar{x}x)y = Xy$ holds. Thus we have $X\leq Xy$ and 
    \begin{align*}
        (Xy)X 
        &= (Xy) (Xx)\tag{Right-absorption: $X = \bar{x}x = (\bar{x}x)x = Xx$}\\
        &= (Xy)x\tag{Flattening}\\
        &= X \tag{Weak closure stability; $(Xy)x = (X\bar{x})x = X$}
    \end{align*}
    Therefore $Xy\leq X$ and by anti-symmetry $X = Xy$ which is what we wanted to show.

    Next we assume that $b\leq ab$ for all $a,b\in M$ and show $X = Y$. Now
    \begin{align*}
        \bar{x}\bar{y} 
        &= (y\bar{x})(yx)\tag{$y\leq \bar{x}, \bar{y} = yx$}\\
        &= (y\bar{x}) x \tag{Flattening} \\
        &= X \tag{$y\leq \bar{x}, X = \bar{x}x$}
    \end{align*}
    and similarly $Y = \bar{y}\bar{x}$. Thus $\bar{y}\leq \bar{x}\bar{y} = X$ and hence 
    \begin{align*}
        XY
        &= X(\bar{y}\bar{x})\tag{$Y = \bar{y}\bar{x}$}\\
        &= X\bar{x} \tag{Flattening; $\bar{y}X = X$}\\
        &= X \tag{Previously: $X\bar{x} = Xy = X$}
    \end{align*}
    Since $Y\leq XY = X$ and similarly $X\leq Y$, it follows by anti-symmetry that $X = Y$.
\end{proof}
    The proof of Theorem \ref{thm: Anti-symmetry equationalized} generalizes the syntactic deduction given by Kowalski in \cite{Kowalski1994}. Interestingly, the Z3-theorem prover \cite{deMouraBjorner2008} is able to produce a proof for both parts in the Theorem \ref{thm: Anti-symmetry equationalized} independently \cite{Forsman2025}, but not for the Theorem \ref{thm: MC-algebras have relative closure term}. A major difference in these settings is the fact that equations being assumed and concluded in Theorem \ref{thm: Anti-symmetry equationalized} are mostly balanced and contain a single function symbol $t$. These equations are not only balanced but left-balanced equation, which are studied in \cite{forsman2024multicategoricalmetatheoremcompletenessrestricted}. These equations have their own restricted, but complete, deduction system, allowing some restrictions on the space for potential proofs.
    \begin{corollary}\label{Algebras forming varieties}
        The classes of MC,MBC,HBC,HBCK-algebras form varieties.
    \end{corollary}
    \begin{proof}
        Note that by Example \ref{Examples of theories}(d), MC-algebras and MCB-algebras coincide and thus the quasi-equation of Right-monotonicity of implication can be replaced by the B-axiom. Theorem \ref{thm: MC-algebras have relative closure term}(2), the term $t(x,y) = (x\to y)\to y$ is a relative closure term in MC-algebras that (strongly) determines the order $x\leq y\coloneqq x\to y\approx e$. Thus the quasi-equation for anti-symmetry can be replaced by $t$-anti-symmetry equation
        \begin{align*}
            (((x\to y)\to y)\to x)\to x\approx (((((x\to y)\to y)\to x)\to x)\to y)\to y.
        \end{align*}\qedhere         
    \end{proof}

\section{Protomodular Algebraic Theories}
This section investigates categorical-algebraic properties, defining pointed, Malcev, protomodular, and semi-abelian theories \cite{BorceuxBourn2004}. We present no-go theorems (Theorem \ref{No-go theorem}) establishing that balanced theories or those admitting either non-discrete monotone or inflationary algebras cannot be Malcev or protomodular. Contrasting with these limitations, we introduce Johnstone algebras over $(*, \to, e)$. These are specifically constructed using Johnstone's protomodular terms to yield a semi-abelian category. The defining axioms ensure $e$ is a right identity and that $x \le y \coloneqq x \to y \approx e$ is reflexive and anti-symmetric. We also consider extensions: adding the Residuation axiom yields a posetal monoid structure. Further adding the Compositionality axiom ensures right-monotonicity for $*$ and $\to$, and links them via a Galois connection ($x*(-) \dashv x \to (-)$), unifying structures like groups and hoops within this semi-abelian framework.

\begin{definition}[Pointed, Malcev and Protomodular]
Let $T$ be an algebraic $\sigma$-theory. We say $T$ is:
    \begin{itemize}
        \item \textbf{Pointed} if the signature $\sigma$ contains a constant symbol $c$ such that for every function symbol $f$ of $\sigma$, $T$ proves the equation $f(c,\ldots, c)\approx c\ (f \approx c,\text{ if $f$ is a constant})$.
        \item \textbf{Malcev} if there exists a ternary term $p(x,y,z)$ such that $T$ proves the equations $p(x,y,y) \approx x$ and $p(x,x,y) \approx y$.
        \item \textbf{Protomodular} if there exist $n\in\N$, an $(n+1)$-ary term $\theta(x_0, x_1, \ldots, x_n)$, and $n$ binary terms $\theta_1(x,y), \ldots, \theta_n(x,y)$ such that $T$ proves the equations:
        \begin{itemize}
            \item $\theta_i(x,x) \approx t_i$ for all $i=1, \ldots, n$, where the terms $t_1,\ldots, t_n$ are constant terms.
            \item $\theta(y, \theta_1(x,y), \ldots, \theta_n(x,y)) \approx x$.
        \end{itemize}
        \item \textbf{Semi-abelian} if it is both pointed and protomodular.
    \end{itemize}
\end{definition}

A Malcev term $p(x,y,z)$ can be defined as $\theta(x,\theta_1(y,z),\ldots, \theta_n(y,z))$ using the protomodular terms $\theta,\theta_i,i\leq n$. Notice that 
 $$
 p(x,y,y) = \theta(x,\theta_1(y,y),\ldots,\theta_n(y,y)) \approx \theta(x,\theta_1(x,x),\ldots, \theta_n(x,x)) \approx x\text{ and }
 $$
 $$
 p(x,x,y) = \theta(x,\theta_1(x,y),\ldots, \theta_n(x,y)) \approx y.
 $$

The Malcev and protomodular properties are preserved under theory extensions. That is, if a theory $T$ has either these terms, any theory $T'\vdash \phi$ for $\phi\in T$, also inherits that property; even the signatures can be extended. This makes it valuable to identify minimal theories possessing these characteristics. Conversely, by contraposition, if a theory $T'$ lacks one of these properties, then any sub-theory $T\subset D_{T'}$ must also lack it, where $D_{T'}$ is the set of equations deduced form $T'$. This principle underlies the utility of the no-go theorems presented next.

\begin{definition}
    Let $T$ be a $\sigma$-theory. Let $M$ be a $T$-algebra. We say that $M$ is
    \begin{itemize}
        \item a \textbf{monotone algebra}, if $M$ admits a partial order $\leq$ such that each structure map is increasing in each component with respect to $\leq$·
        \item a \textbf{partially inflationary algebra}, if $M$ admits an order $\leq$ so that no constant is the minimum element of $M$ and that each structure map $\phi\colon M^n\to M$, $n>0$, is inflationary in some component: There is $i\leq n$ so that $x_i\leq \phi(x_1,\ldots, x_n)$ for each $x_1,\ldots, x_n\in M$.
        \item An \textbf{inflationary algebra}, if $M$ admits a partial inflationary structure $\leq$ where $M(c)$ is the maximum element of $M$ for each constant symbol $c$ of $\sigma$.
    \end{itemize}
\end{definition}

\begin{theorem}[No-Go Theorems for Malcev and Protomodular Theories]\label{No-go theorem}
Let $T$ be an algebraic $\sigma$-theory. Then the following assertions hold:

\begin{enumerate}
    \item \textbf{Balanced Theories:} If $T$ is a balanced theory, then $T$ cannot be Malcev.

    \item \textbf{Monotone Algebras:} Let $M$ be a monotone $T$-algebra and assume that $T$ is Malcev. Then $M\vDash x\approx y$.

    \item \textbf{Inflationary Algebras:} Suppose $T$ admits a partially inflationary $T$-algebra $M$. If $T$ is protomodular, then $M\vDash x\approx y$. Furthermore, if, $M$ is an inflationary algebra and $T$ is Malcev, then $M\vDash x\approx y$.
\end{enumerate}
\end{theorem}

\begin{proof}
\hfill
\begin{enumerate}
    \item Assume $T$ is balanced. The class of balanced equations is closed under the deduction rules (reflexivity, symmetry, transitivity, substitution). Therefore, any equation provable from $T$ must be balanced. If $T$ were Malcev, there would exist a term $p(x,y,z)$ such that $T \vdash p(x,y,y) \approx x$ and $T \vdash p(x,x,y) \approx y$. Since these derived equations must be balanced, the variable $y$ cannot be expressed in the term $p(x,y,z)$.  This yields $T \vdash x \approx y$ and we have reached a contradiction, since $x\approx y$ is not a balanced equation. 
    
    \item Assume $T$ admits such an algebra $M$. By structural induction on terms, any $\sigma$-term $t(x_1, \ldots, x_n)$ induces a function $M(t)\colon M^n \to M$ that is increasing in all components. If $T$ has a Malcev term $p(x,y,z)$, then $M(p)\colon M^3\to M$ would be increasing. Since $\leq$ is non-discrete, let $a, b \in M$ with $a < b$. Then $a \leq b$ implies:
    \[ b = M(p)(a, a, b) \leq M(p)(a, b, b) = a \]
    By the anti-symmetry of $\leq$, this forces $a = b$, contradicting $a < b$. Thus, $T$ cannot be Malcev.

    \item Let $M$ be a partially inflationary $T$-algebra.
    \textit{(Protomodularity Part):} Let $V$ be the set of variables. For a term $t$, let $M(t): M^V \to M$ be the induced function. First, by structural induction on terms we have that $M(t)\leq M(x)$ implies that $M(t) = M(x)$ for every $\sigma$-term $t$ and variable $x\in V$. Second, by structural induction, we attain that $M(t(t_1,\ldots, t_n))\leq M(x)$ implies that $M(t_i) = M(x)$ for some $i\leq n$ and $\sigma$-terms $t(x_1,\ldots, x_n),t_1,\ldots, t_n$. 
    
    Now, assume $T$ is protomodular with terms $\theta, \theta_1, \ldots, \theta_n$. Since
    $$
    \theta(y, \theta_1(x,y), \ldots, \theta_n(x,y)) \approx x
    $$
    holds in $M$, we have that either $M(y) = M(x)$ or $M(\theta_i(x,y)) = M(x)$ for some $i\leq n$. Both cases show that $M$ satisfies $x\approx y$.

    (Malcev part): Assume constant symbols of $\sigma$ are interpreted as the maximum element of $M$. An easy induction shows that for any term $t$, there is a variable $x$ so that $M(x)\leq M(t)$. If $T$ has a Malcev term $p(x,y,z)$, then $M(x)\leq M(p(x_1,x_2,x_3))$. A small case analysis then shows that $M\vDash x\approx y$.\qedhere

\end{enumerate}
\end{proof}

\begin{example}[Theories Precluded from being Malcev]\label{ex:nogo_examples}
Theorem \ref{No-go theorem} provides criteria that prevent an algebraic theory from being Malcev (unless trivial). The following table categorizes several classes of algebras whose algebraic theories cannot contain Malcev terms by the criterion from Theorem \ref{No-go theorem}:

\medskip

\noindent 
\begin{tabular}{|p{0.3\textwidth}|p{0.3\textwidth}|p{0.3\textwidth}|}
\hline
\textbf{Balanced Algebras} \par \textit{(via Thm \ref{No-go theorem}(1))} &
\textbf{Monotone Algebras} \par \textit{(via Thm \ref{No-go theorem}(2))} &
\textbf{Inflationary Algebras} \par \textit{(via Thm \ref{No-go theorem}(3))} \\
\hline \hline 

\begin{itemize} \itemsep1pt 
    \item Monoids
    \item Pointed semilattices
    \item Heyting semilattices (Ex. \ref{Examples of theories}(c)) without assuming Reflexivity: $x\to x\approx e$
    \end{itemize}
&
\begin{itemize} \itemsep0pt
    \item Lattices $(\land, \lor)$
    \item Bounded Lattices $(\land,\lor, \top,\bot)$
    \item Lattice-ordered monoids $(*,\land,\lor,\top,\bot)$
\end{itemize}
&
\begin{itemize} \itemsep0pt
    \item HBCK-algebras
    \item MCK-algebras
    \item Any non-collapsing $(\lor ,\to,\top)$-theory admitting an order and proving
    \begin{itemize}
        \item $x\leq x\lor y$
        \item $y\leq x\to y$
        \item $x\leq \top$
    \end{itemize}
\end{itemize}
\\
\hline
\end{tabular}
\end{example}
\begin{definition}[Johnstone Algebras]\label{def:johnstone_algebra}
Consider the signature $\sigma = (*, \to, e)$ with arities $(2, 2, 0)$. The theory $J$ of \emph{Johnstone algebras} is defined by the following equations:
\begin{enumerate}
    \item $x\to x \approx e$ \hfill (Reflexivity)
    \item $e\to x \approx x$ \hfill (Unit-reduction)
    \item $((x\to y)\to y)* (((x\to y)\to y)\to x) \approx x$ \hfill (Monotone division)
\end{enumerate}
These axioms are chosen specifically so that the \emph{Johnstone protomodular terms}, given by
\begin{align*}
    \theta(x_0, x_1, x_2) &\coloneqq (x_1 \to x_0) * x_2 \\
    \theta_1(x,y) &\coloneqq x\to y \\
    \theta_2(x,y) &\coloneqq ((x\to y)\to y)\to x,
\end{align*}
satisfy the protomodularity equations, ensuring the resulting category is semi-abelian (see Theorem \ref{thm:J_properties}).
\end{definition}

As noted previously (Theorem \ref{thm: characterizing reflectivity and transitivity}(1)), the Monotone division axiom can be interpreted as an equational version of the quasi-equation $x\leq y \implies y*(y\to x) \approx x$, particularly when Modus ponens ($x \le (x\to y)\to y$) holds. Furthermore, in theories that also satisfy Residuation, Monotone division induces the Monotonic exchange (M-axiom (21)): $(t(x,y)\to x)\to (t(x,y)\to z) \approx x\to z$, where $t(x,y)=(x\to y)\to y$.

The following theorem details the fundamental properties derived from the Johnstone algebra axioms.

\begin{theorem}[Properties of Johnstone Algebras]\label{thm:J_properties}
Let $J$ be the theory of Johnstone algebras. Define the relation $x\leq y$ by the equation $x\to y\approx e$. The following properties hold:
\begin{enumerate}
    \item \textbf{(Semi-abelian Structure)} The theory $J$ is semi-abelian. Specifically:
        \begin{itemize}
            \item $J$ is pointed with respect to the constant $e$.
            \item $J$ proves that $e$ is a right identity for $*$, i.e., $J \vdash x*e \approx x$.
            \item The Johnstone protomodular terms $\theta, \theta_1, \theta_2$ satisfy the protomodularity equations relative to $e$:
                \begin{align*} J &\vdash \theta_1(x,x) \approx e, \\ J &\vdash \theta_2(x,x) \approx e, \\ J &\vdash \theta(y, \theta_1(x,y), \theta_2(x,y)) \approx x. \end{align*}
        \end{itemize}
    \item \textbf{(Order Properties)} The relation $\leq$ is reflexive and anti-symmetric in any $J$-algebra. Furthermore, $\leq$ is semi-determined by the term $t(x,y) = (x\to y)\to y$, meaning $J \vdash (x \leq y \implies t(x,y) \approx y)$.
\end{enumerate}
\end{theorem}

\begin{proof}
\hfill
\begin{enumerate}
    \item \textbf{(Semi-abelian Structure)}
        \begin{itemize}
            \item Pointedness: We need $e\to e \approx e$ and $e*e \approx e$. The Reflexivity axiom gives $e\to e \approx e$. For $e*e \approx e$, we first derive the right identity property $x*e \approx x$:
            \begin{align*}
                x &\approx ((x\to x)\to x)*(((x\to x)\to x)\to x) \tag{Monotone division} \\
                  &\approx (e\to x)*((e\to x)\to x) \tag{Reflexivity} \\
                  &\approx x * (x\to x) \tag{Unit-reduction} \\
                  &\approx x * e \tag{Reflexivity}
            \end{align*}
            Substituting $e$ on $x$, we have $e \approx e*e$. Thus $J$ is pointed.
            \item Protomodularity Equations:
                \begin{itemize}
                    \item $\theta_1(x,x) = x\to x \approx e$ by Reflexivity.
                    \item $\theta_2(x,x) = ((x\to x)\to x)\to x = (e\to x)\to x = x\to x \approx e$ by Reflexivity and Unit-reduction.
                    \item $\theta(y, \theta_1(x,y), \theta_2(x,y)) = (\theta_1(x,y) \to y) * \theta_2(x,y)$ \\
                       $\hphantom{\theta(y, \theta_1(x,y), \theta_2(x,y))} = ((x\to y) \to y) * (((x\to y)\to y)\to x)$ \\
                       $\hphantom{\theta(y, \theta_1(x,y), \theta_2(x,y))} \approx x$ by Monotone division.
                \end{itemize}
        \end{itemize}
        Since $J$ is pointed and the Johnstone terms satisfy the protomodularity equations, $J$ is a semi-abelian theory.

    \item \textbf{(Order Properties)}
        Reflexivity ($x \leq x$) follows directly from the Reflexivity axiom.
        For anti-symmetry, assume $a \leq b$ and $b \leq a$ hold (i.e., $a\to b \approx e$ and $b\to a \approx e$). Now
        \begin{align*}
            a &\approx ((a\to b)\to b)*(((a\to b)\to b)\to a) \tag{Monotone division} \\
              &\approx b * (b \to a) \tag{$a\leq b$ and Unit-reduction} \\
              &\approx b * e \tag{$b\leq a$} \\
              &\approx b \tag{Right identity}
        \end{align*}
        For semi-determination by $t(x,y)=(x\to y)\to y$: Assume $x \leq y$, so $x \to y \approx e$. Then $t(x,y) = (x\to y)\to y \approx e \to y \approx y$ by Unit-reduction.\qedhere
\end{enumerate}
\end{proof}

As illustrated in Example \ref{Examples of theories}(e), significant classes of algebras such as hoops and left loops satisfy the axioms of Johnstone algebras and therefore form semi-abelian categories.
\begin{theorem}[Properties of Enhanced Johnstone Algebras]\label{Properties of Johnstone algebras}
Let $(X,*,\to,e)$ be a Johnstone algebra, with $\leq$ defined by $x\leq y \iff x\to y \approx e$. Consider the effects of adding standard axioms from related algebraic structures:
\begin{enumerate}
    \item \textbf{(Residuation)} If $X$ satisfies the Residuation axiom, then $X$ defines a poset and a monoid structure. A left loop satisfying Residuation is precisely a group.
    \item \textbf{(Compositionality)} If $X$ satisfies the Compositionality axiom, then $(X, \leq)$ is a poset, and the implication $\to$ is increasing in its second argument. Consequently, a left loop satisfying Compositionality is precisely a group.
    \item \textbf{(Residuation and Compositionality)} If $X$ satisfies both Residuation and Compositionality, then for each $x \in X$, the operations $x*(-)$ and $x\to (-)$ form a Galois connection, i.e., $x*y \leq z \iff y \leq x\to z$. This implies that multiplication $*$ is increasing in its right argument.
    \item \textbf{(Commutativity, Residuation, and Compositionality)} If $X$ is commutative and satisfies Residuation and Compositionality, then $X$ also satisfies Implicative Commutativity ($x\to (y\to z)\approx y\to (x\to z)$) and Modus ponens ($x \to ((x\to y)\to y) \approx e$). Furthermore, multiplication $*$ is increasing in both arguments, and implication $\to$ is decreasing in the first argument and increasing in the second. Hoops and commutative groups are primary examples of such algebras.
\end{enumerate}
\end{theorem}

\begin{proof}
\hfill
\begin{enumerate}
    \item Assume $X$ satisfies Residuation. Residuation implies Monotonic exchange (M-axiom (21)). The proof of Theorem \ref{thm: MC-algebras have relative closure term}(1) demonstrated that Monotonic Exchange (in the presence of Reflexivity and the established anti-symmetry of $\le$) implies transitivity of $\leq$. Thus $(X,\le)$ is a poset.
    
    The equivalence $x*w \leq y \iff w \leq x\to y$ is implied by Residuation.
    
    Consider the following Yoneda-style arguments showing Associativity and Left identity: For all $x,y,z,w \in X$:
    \begin{align*}
    &(x*y)*z \leq w                       &e*x&\leq w\\
    &\iff z \leq (x*y)\to w          &\iff x&\leq e\to w\\
    &\iff z \leq y\to(x\to w)       &\iff x&\leq w\\
    &\iff y*z \leq x\to w\\
    &\iff x*(y*z) \leq w 
    \end{align*}
    Since this holds for all $w$, we have by anti-symmetry $(x*y)*z \approx x*(y*z)$ and $e*x \approx x$. Since $e$ is already a right identity (Theorem \ref{thm:J_properties}), $(X,*,e)$ is a monoid.
    
    A left loop that is associative is a group. Conversely, any group $(G, *, e)$ with $x \to y \coloneqq x^{-1}*y$ satisfies the Johnstone axioms and Residuation: 
    $$
    x\to (y\to z) = x^{-1}*(y^{-1}*z) = (x^{-1}*y^{-1})*z = (y*x)^{-1}*z = (y*x)\to z.
    $$

    \item Assume $X$ satisfies Compositionality. To show transitivity of $\leq$, let $x\leq y$ and $y\leq z$. Now
    \begin{align*} x\to z &\approx e\to(e \to (x\to z)) \tag{Unit-reduction} \\
    &\approx (y\to z)\to ((x\to y) \to (x\to z)) \tag{$y\le z, x\leq y$}\\
    &\approx e\tag{Compositionality}
    \end{align*}
    
    Thus $x \leq z$. To show $\to$ is increasing in the second argument, assume $y \leq z$.
    \begin{align*} (x\to y)\to (x\to z)
    &\approx e \to ((x\to y)\to (x\to z)) \tag{Unit-reduction}\\
    &\approx (y\to z)\to ((x\to y)\to (x\to z)) \tag{$y\leq z$}\\
    &\approx e \tag{Compositionality} \end{align*}    
    Let $G$ be a group with $x,y,z\in G$. Since $G$ is a group, we have $x\to y = x^{-1}*y$ and
    $$
    (x\to y)\to ((z\to x)\to (z\to y)) = y^{-1}xx^{-1}zz^{-1}y = e
    $$
    and thus $G$ is compositional. Assume that $(X,*,\to,e)$ is a compositional left loop. Since the posetal structure is discrete, we have an equation $x\to y = (z\to x)\to (z\to y)$ for any $x,y,z\in X$. It suffices to show Residuality by part $1$. For $x,y,z\in X$ we have
    $$
    yx\to z = (y\to yx)\to (y\to z) = x\to (y\to z).
    $$

    \item Assume $X$ is residual and compositional. By part (2), $\to$ is increasing in the second argument. The definition of Residuation $x \to (y\to z) \approx (y*x) \to z$ which implies that $a\leq b\to c \iff b*a\leq c$ for each $a,b\in X$, and so we have the Galois connection $a*(-) \dashv a \to (-)$ for each $a\in X$. Thus $a*(-)$ is increasing for each $a\in X$.

    \item Assume $X$ is commutative, residual and compositional. For $x,y,z\in X$:
    \[
    x\to (y\to z) \approx (y*x)\to z \approx (x*y)\to z \approx y\to (x\to z)
    \]
    using Residuation, Commutativity, Residuation. This proves Implicative Commutativity. Implicative commutativity (with Reflexivity) implies Modus ponens.
    
    Monotonicity of $*$ in the right argument follows from part (3). By commutativity, it is also monotone in the left argument. The Galois connection $x*(-) \dashv x\to (-)$ for every $x\in X$ implies $\to$ is decreasing in the first argument. The decreasing in the first component can be deduced from the implicative commutativity as it was done in Theorem \ref{thm: Anti-symmetry equationalized}.\qedhere
\end{enumerate}
\end{proof}
The RBJ-algebras, Johnstone algebras satisfying both Residuation and Compositionality, form a semi-abelian theory whose algebras $M$ possess an equationally defined poset structure which almost forms a monoidal closed structure ($x*(-)\dashv x\to(-)$) for $x\in M$. The structure fails to be fully monoidal closed only because multiplication might not be increasing in the left argument. Adding commutativity remedies this, yielding the class of RBCJ-algebras (Commutative, Compositional, Residual Johnstone algebras), which indeed induce symmetric monoidal closed posets within the context of a semi-abelian category.
\section*{Conclusion}

This paper introduced \emph{Johnstone algebras}, defined over $(*, \to, e)$ and axiomatized by Reflexivity, Unit-reduction, and Monotone division. This structure is designed to ensure the associated category is semi-abelian, following Johnstone's work on Heyting semilattices \cite{Johnstone2004}. Johnstone algebras provide a unified semi-abelian framework encompassing algebraic structures like left loops and groups alongside logical ones like hoops and Heyting semilattices.

Our primary technical contributions are twofold. First, we defined the notion of a weak relative closure term $t(x,y)$ in a theory $T$ and established in Theorem \ref{thm: Anti-symmetry equationalized} that the equation
$$
t(t(x,y),x)\approx t(t(t(x,y),x),y)
$$
is equivalent to the quasi-equation for anti-symmetry proving the quasi-variety of MBC-algebras form a variety. Second, we presented no-go theorems demonstrating that no algebraic theory $T$ admitting a partially ordered, non-discrete, inflationary/monotone algebra can have a Malcev term. Especially, no $(\to,e)$-theory can be both Malcev and admit a non-discrete partially ordered algebra $M$ satisfying $y\leq x\to y$ and $x\leq e$ for each $x,y\in M$, thereby contextualizing the specific structure of Johnstone algebras required for their semi-abelian property.

Further analysis of Johnstone algebras with Residuation and Compositionality revealed connections to residuated structures and monoidal closed categories. Within this richer setting, the Monotone division axiom, via its equivalent form Monotonic exchange, can be interpreted as a restricted version of the H-axiom, linking it to established logical principles. Preliminary observations also suggest connections between Monotone division, categorical evaluation, and duality concepts, warranting deeper study.

In essence, Johnstone algebras offer a framework integrating algebraic symmetry and logical implication within a categorically well-behaved setting. While the base theory secures semi-abelianness, extensions with Compositionality and Residuation provide promising directions for investigating novel algebraic and logical systems that bridge these domains.

\section*{Acknowledgements}
I would like to offer my sincere gratitude to my supervisor, Professor Tim Van der Linden, for his comments on refining this paper. I am also grateful for the Belgian National Fund for Scientific Research (FNRS) for the financial support through the ASP-grant.
\printbibliography

\end{document}